\documentclass{amsart}

\usepackage[english]{babel}
\usepackage[utf8x]{inputenc}
\usepackage[T1]{fontenc}
\usepackage{comment}
\usepackage[none]{hyphenat}
\usepackage{adjustbox}
\usepackage{tikz}
\usetikzlibrary{arrows, decorations.markings}
\usepackage{ytableau}
\usepackage{mathtools}
\usepackage{cite}

\usepackage{verbatim}

\usepackage{comment}

\usepackage{amsmath}
\usepackage{enumitem}  
\usepackage{amssymb}
\usepackage{amsthm}
\usepackage{graphicx}
\usepackage[colorinlistoftodos]{todonotes}
\usepackage[colorlinks=true, allcolors=blue]{hyperref}
\usepackage{url}

\newcommand{\RNum}[1]{\uppercase\expandafter{\romannumeral #1\relax}}

\newcommand{\la}{\lambda}

\usepackage{mathtools}

\DeclarePairedDelimiter\floor{\lfloor}{\rfloor}

\usepackage{mathrsfs}

\newtheorem{thm}{Theorem}

\theoremstyle{definition}

\newtheorem{lem}[thm]{Lemma}

\newtheorem{prop}[thm]{Proposition}
\newtheorem{rem}[thm]{Remark}

\newtheorem{ex}[thm]{Example}
\allowdisplaybreaks

\newcommand{\ols}[1]{\mskip0\thinmuskip\overline{\mskip-.5\thinmuskip {#1} \mskip-2.5\thinmuskip}\mskip0\thinmuskip}

\numberwithin{thm}{section}

\title[bar-cores, CSYDs, and doubled distinct cores]
{Results on bar-core partitions, core shifted Young diagrams, and doubled distinct cores}

\author{Hyunsoo Cho}
\address{Hyunsoo Cho, Institute of Mathematical Sciences, Ewha Womans University, Seoul, Republic of Korea}
\email{hyunsoo@ewha.ac.kr}

\author{JiSun Huh}
\address{JiSun Huh, Department of Mathematics, Ajou University, Suwon, Republic of Korea}
\email{hyunyjia@ajou.ac.kr}

\author{Hayan Nam}
\address{Hayan Nam, Department of Mathematics, Duksung Women's University,  Seoul, Republic of Korea}
\email{hnam@duksung.ac.kr}

\author{Jaebum Sohn}
\address{Jaebum Sohn, Department of Mathematics, Yonsei University, Seoul, Republic of Korea}
\email{jsohn@yonsei.ac.kr}

\begin{document}

\begin{abstract}
Simultaneous bar-cores, core shifted Young diagrams (or CSYDs), and doubled distinct cores have been studied since Morris and Yaseen introduced the concept of bar-cores. In this paper, our goal is to give a formula for the number of these core partitions on $(s,t)$-cores and $(s,s+d,s+2d)$-cores for the remaining cases that are not covered yet.
In order to achieve this goal, we observe a characterization of $\overline{s}$-core partitions to obtain characterizations of doubled distinct $s$-core partitions and $s$-CSYDs.
By using them, we construct $NE$ lattice path interpretations of these core partitions on $(s,t)$-cores. Also, we give free Motzkin path interpretations of these core partitions on $(s,s+d,s+2d)$-cores.

\end{abstract}

\maketitle
\sloppy


\section{Introduction}
A \emph{partition} $\la = (\la_1, \la_2, \ldots, \la_{\ell})$ of $n$ is a non-increasing positive integer sequence whose sum of the parts $\la_i$ is $n$. We denote that $\la_i \in \la$ and visualize a partition $\la$ with the \emph{Young diagram} $D(\la)$. For a partition $\la$, $\la'$ is called the \emph{conjugate} of $\la$ if $D(\la')$ is the reflection of $D(\la)$ across the main diagonal, and $\la$ is called \emph{self-conjugate} if $\la=\la'$. An $(i,j)$-box of $D(\la)$ is the box at the $i$th row from the top and the $j$th column from the left. The \emph{hook length} of an $(i,j)$-box, denoted by $h_{i,j}(\la)$, is the total number of boxes on the right and the below of the $(i,j)$-box and itself, and the \emph{hook set} $\mathcal{H}(\la)$ of $\la$ is the set of hook lengths of $\la$. We say that a partition $\la$ is an \emph{$s$-core} if $ks\notin\mathcal{H}(\la)$ for all $k \in \mathbb{N}$ and is an \emph{$(s_1, s_2, \dots, s_p)$-core} if it is an $s_i$-core for all $i=1,2,\dots,p$.
Figure \ref{fig:ex} illustrates the Young diagram of a partition and a hook length.

\begin{figure}[ht!]
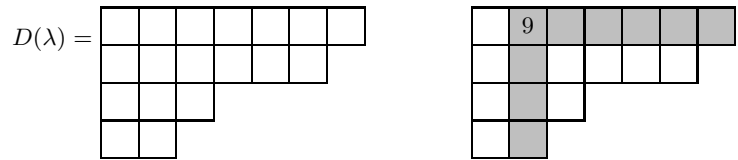

\centering
\small{
$D(\la)=$~\begin{ytableau}
~&~&~&~&~&~&~ \\
~&~&~&~&~&~ \\
~&~&~ \\
~&~
\end{ytableau}
\qquad \qquad
\begin{ytableau}
~&*(gray!50)9&*(gray!50)&*(gray!50)&*(gray!50)&*(gray!50)&*(gray!50) \\
~&*(gray!50)&~&~&~&~ \\
~&*(gray!50)&~ \\
~&*(gray!50)
\end{ytableau}}

\caption{The Young diagram of the partition $\la=(7,6,3,2)$ and a hook length $h_{1,2}(\la)=9$.} \label{fig:ex}
\end{figure}

There have been active research on the number of simultaneous core partitions and self-conjugate simultaneous core partitions since Anderson \cite{Anderson} counted the number of $(s,t)$-core partitions for coprime $s$ and $t$. For more information, see \cite{AL,FMS,Wang} for example. 
In this paper, we investigate the three different types of core partitions, which are called bar-core partitions, core shifted Young diagrams, and doubled distinct core partitions. Researchers have been studied them independently but they are inevitably related to each other.

We first give the definitions of the three objects that we only deal with under the condition that the partition is \emph{strict}, which means that each part is all distinct.

For a strict partition $\la=(\la_1, \la_2, \ldots, \la_{\ell})$, 
an element of the set 
\[
\{\la_i+\la_{i+1}, \la_i+\la_{i+2}, \dots, \la_i+\la_{\ell} \} \cup \left( \{ \la_{i}, \la_{i}-1, \dots, 1 \} \setminus \{\la_{i}-\la_{i+1}, \dots, \la_{i}-\la_{\ell}\} \right)
\]
is called a \emph{bar length} in the $i$th row. A strict partition $\la$ is called an \emph{$\overline{s}$-core} (\emph{$s$-bar-core}) if $s$ is not a bar length in any row in $\la$. For example, the sets of bar lengths in every row of $\la=(7,6,3,2)$ are $\{13,10,9,7,6,3,2\}$, $\{9,8,6,5,2,1\}$, $\{5,3,2\}$, and $\{2,1\}$. Thus, $\la$ is an $\overline{s}$-core partition for $s=4,11,12$, or $s\geq 14$.

The \emph{shifted Young diagram} $S(\la)$ of a strict partition $\la$ is obtained from $D(\la)$ by shifting the $i$th row to the right by $i-1$ boxes for each $i$.
The \emph{shifted hook length} $h^*_{i,j}(\la)$ of an $(i,j)$-box in $S(\la)$ is the number of boxes on its right, below and itself, and the boxes on the $(j+1)$st row if exists. For example, the left diagram in Figure \ref{fig:bar} shows the shifted Young diagram of the partition $(7,6,3,2)$ with the shifted hook lengths. The shifted hook set $\mathcal{H}^*(\la)$ is the set of shifted hook lengths in $S(\la)$. A shifted Young diagram $S(\la)$ is called an \emph{$s$-core shifted Young diagram}, shortly $s$-CSYD, if none of the shifted hook lengths of $S(\la)$ are divisible by $s$. Sometimes we say that ``$\la$ is an $s$-CSYD'' instead of ``$S(\la)$ is an $s$-CSYD''.

Given a strict partition $\la=(\la_1, \la_2, \ldots, \la_{\ell})$, the \emph{doubled distinct partition} of $\la$, denoted by $\la \la$, is a partition whose Young diagram $D(\la \la)$ is defined by adding $\la_i$ boxes to the $(i-1)$st column of $S(\la)$.
In other words, the Frobenius symbol of the doubled distinct partition $\la\la$ is given by
\[
\begin{pmatrix}
\la_1 & \la_2 & \cdots &\la_{\ell}\\
\la_1 -1 & \la_2 -1 & \cdots & \la_{\ell} -1
\end{pmatrix}.
\]
The doubled distinct partition $\la\la$ is called a \emph{doubled distinct $s$-core} if none of the hook lengths are divisible by $s$.
Note that the hook set of $D(\la\la)$ that is located on the right of the main diagonal is the same as $\mathcal{H}^*(\la)$. Indeed, the hook lengths on the $(\ell+1)$st column of $D(\la\la)$ are the parts of $\la$ and the deletion of this column from $D(\la\la)$ gives a self-conjugate partition. See Figure \ref{fig:bar} for example.


\begin{figure}[ht!]
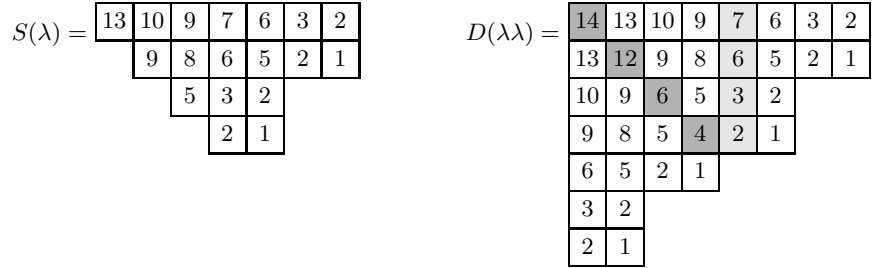

{\small
$S(\la)=~$\begin{ytableau}
13&10&9&7&6&3&2 \\
\none&9&8&6&5&2&1 \\
\none&\none&5&3&2 \\
\none&\none&\none&2&1 \\
\end{ytableau}
\qquad \qquad
$D(\la\la)=~$\begin{ytableau}
*(gray!60)14&13&10&9&*(gray!20)7&6&3&2 \\
13&*(gray!60)12&9&8&*(gray!20)6&5&2&1 \\
10&9&*(gray!60)6&5&*(gray!20)3&2 \\
9&8&5&*(gray!60)4&*(gray!20)2&1 \\
6&5&2&1 \\
3&2 \\
2&1
\end{ytableau}}
\caption{The shifted Young diagram $S(\la)$ with the shifted hook lengths and the doubled distinct partition $\la\la$ with the hook lengths for the strict partition $\la=(7,6,3,2)$.}\label{fig:bar}
\end{figure}

We extend the definition of simultaneous core partitions to bar-core partitions and CSYDs. We use the following notations for the variety sets of core partitions,
\begin{align*}
\mathcal{SC}_{(s_1, s_2, \dots, s_p)} &: \text{~the set of self-conjugate $(s_1, s_2, \dots, s_p)$-cores},\\ 
\mathcal{BC}_{(s_1, s_2, \dots, s_p)} &: \text{~the set of  $(\overline{s_1}, \overline{s_2},\dots, \overline{s_p})$-cores},\\ 
\mathcal{CS}_{(s_1, s_2, \dots, s_p)} &: \text{~the set of $(s_1, s_2, \dots, s_p)$-CSYDs},\\ 
\mathcal{DD}_{(s_1, s_2, \dots, s_p)} &: \text{~the set of doubled distinct $(s_1, s_2, \dots, s_p)$-cores}. 
\end{align*}

There are a couple of results on counting the number of simultaneous core partitions of the three objects, bar-cores, CSYDs, and doubled distinct cores. Bessenrodt and Olsson \cite{BO} adopted the Yin-Yang diagram to count the number of $(\ols{s\phantom{t}},\overline{t})$-core partitions for odd numbers $s$ and $t$, Wang and Yang \cite{WY} counted the same object when $s$ and $t$ are in different parity, and Ding \cite{Ding} counted the number of $(s,s+1)$-CSYDs (as far as the authors know these are the only counting results on the three objects known until now). Our main goal is to fill out all the possible results we could get on $(s,t)$-cores and $(s,s+d,s+2d)$-cores for the three objects by constructing some bijections. Additionally, we hire a well-known object so called self-conjugate core partitions to enumerate the number of such core partitions. For instance, bar-core partitions and self-conjugate core partitions are related to each other; Yang \cite[Theorem 1.1]{Yang} constructed a bijection between the set of self-conjugate $s$-cores and that of $\overline{s}$-cores for odd $s$; Gramain, Nath, and Sellers \cite[Theorem 4.12]{GNS} gave a bijection between self-conjugate $(s,t)$-core partitions and $(\ols{s\phantom{t}},\overline{t})$-core partitions, where both $s$ and $t$ are coprime and odd. 

The following theorems are the main results in this paper. 

\begin{thm}\label{thm:main1}
For coprime positive integers $s$ and $t$, the number of doubled distinct $(s,t)$-core partitions is
\[
|\mathcal{DD}_{(s,t)}|=\binom{\lfloor (s-1)/2 \rfloor + \lfloor (t-1)/2 \rfloor}{\lfloor (s-1)/2 \rfloor},
\]
and the number of $(s,t)$-CSYDs is
\[
|\mathcal{CS}_{(s,t)}|=\binom{\floor*{(s-1)/2} + \floor*{t/2} -1}{\floor*{(s-1)/2}}  +\binom{\floor*{s/2} + \floor*{(t-1)/2}-1}{\floor*{(t-1)/2}}.
\]
\end{thm}


\begin{thm}\label{thm:unifying}
Let $s$ and $d$ be coprime positive integers.
\begin{enumerate}
\item[(a)] For odd $s$ and even $d$, 
\begin{align*}
|\mathcal{BC}_{(s,s+d,s+2d)}|&=|\mathcal{CS}_{(s,s+d,s+2d)}|=|\mathcal{DD}_{(s,s+d,s+2d)}|\\
&=\sum_{i=0}^{(s-1)/2}\binom{(s+d-3)/2}{\lfloor i/2 \rfloor}\binom{(s+d-1)/2-\lfloor i/2 \rfloor}{(s-1)/2-i}.
\end{align*}
\item[(b)] For odd numbers $s$ and $d$, 
\begin{align*}
&|\mathcal{BC}_{(s,s+d,s+2d)}|=|\mathcal{CS}_{(s,s+d,s+2d)}|\\
&~~=\sum_{i=0}^{(s-1)/2}\binom{(d-1)/2+i}{\lfloor i/2 \rfloor}\left( \binom{(s+d-2)/2}{(d-1)/2+i} + \binom{(s+d-4)/2}{(d-1)/2+i}\right).
\end{align*}

\item[(c)] For even $s$ and odd $d$,
\begin{align*}
|\mathcal{BC}_{(s,s+d,s+2d)}|=&\sum_{i=0}^{s/2} \binom{(s+d-1)/2}{\lfloor i/2 \rfloor, \lfloor (d+i)/2\rfloor, s/2 -i},  \\
|\mathcal{CS}_{(s,s+d,s+2d)}|=&\sum_{i=0}^{(s-2)/2}\binom{(s+d-3)/2}{\lfloor i/2 \rfloor}\binom{(s+d-3)/2-\lfloor i/2 \rfloor}{(s-2)/2-i}\\
&+\sum_{i=0}^{(s-2)/2}\binom{(s+d-5)/2}{\lfloor i/2 \rfloor}\binom{(s+d-1)/2-\lfloor i/2 \rfloor}{(s-2)/2-i}.
\end{align*}

\item[(d)] For odd $d$, 
\[
|\mathcal{DD}_{(s,s+d,s+2d)}|=\sum_{i=0}^{ \lfloor(s-1)/2\rfloor} \binom{\lfloor (s+d-2)/2\rfloor }{\lfloor i/2 \rfloor, \lfloor (d+i)/2\rfloor,  \lfloor(s-1)/2\rfloor  -i}.
\]
\end{enumerate}
\end{thm}

This paper is organized as follows: In Section \ref{sec:2}, we obtain useful propositions involving the three objects which are used frequently throughout this paper. Restricted those objects by the size of partitions, we get the generating functions of $\overline{s}$-cores and $s$-CSYDs for even $s$. 
Section \ref{sec:double} includes connections between the sets of $NE$ lattice paths and the three objects with the condition being $(s,t)$-cores. We consider the Yin-Yang diagrams to find the number of doubled distinct $(s,t)$-core partitions and the number of $(s,t)$-CSYDs by constructing each bijection to a certain set of $NE$ lattice paths. 
In Section \ref{sec:triple}, we describe the relations between free Motzkin paths and the three objects under the condition of being $(s,s+d,s+2d)$-cores by using the $(\overline{s+d},d)$-abacus diagram, the $(\overline{s+d},d)$-abacus function, and their properties. From the bijections we set up, we count the number of each $(s,s+d,s+2d)$-core partitions as a result of the number of corresponding free Motzkin paths.


\section{Properties and generating functions}\label{sec:2}

 We begin this section by showing a property which follows straightly from the definition of the bar lengths and the shifted hook lengths. 

\begin{lem}\label{lem:barhook}
Let $\la = (\la_1, \la_2, \dots, \la_{\ell})$ be a strict partition. The set of bar lengths in the $i$th row of $\la$ is equal to the set of the shifted hook lengths in the $i$th row of $S(\la)$.
\end{lem}

\begin{proof}
Let $\mu \coloneqq (\la_1 - \ell +1, \la_2 -\ell +2, \dots, \la_{\ell})$.
By the definition of the shifted hook lengths, we have
\[
h_{i,j}^*(\la)=\begin{cases}
\la_i+\la_{j+1} & \text{ if }~ i \le j \le \ell-1,\\
h_{i, j-\ell+1}(\mu) & \text{ if }~ \ell \le j \le \la_i.
\end{cases}
\]
We show that the statement is true for the first row.
Assume, on the contrary, that $h_{1,j}^*(\la)=h_{1, j-\ell+1}(\mu)=\la_1-\la_k=h_{1,1}(\mu)-h_{k,1}(\mu)$ for some $k$. Then, by the definition of hook lengths, 
\[
\mu_1+\mu_{j-\ell+1}'-(j-\ell+1) = (\mu_1+\mu_1'-1)-(\mu_k+\mu_1' -k),
\]
which implies that $\mu_k+\mu_{j-\ell+1}'-(k+j-\ell)=h_{k, j-\ell+1}(\mu)=0$. Since the hook lengths are always nonzero, we get a contradiction. Similarly, this argument works for the $i$th row in general. 
\end{proof}

\subsection{Characterizations}

In the theory of core partitions, a partition $\la$ is an $s$-core if $s\notin \mathcal{H}(\la)$ or, equivalently, if  $ms\notin\mathcal{H}(\la)$ for all $m$. In \cite[p. 31]{MY}, Morris and Yaseen gave a corollary that $\la$ is an $\overline{s}$-core if and only if none of the bar lengths in the rows of $\la$ are divisible by $s$. However, Olsson \cite[p. 27]{Olsson-book} pointed out that this corollary is not true when $s$ is even. In Figure \ref{fig:bar}, one can see that $\la=(7,6,3,2)$ is a $\overline{4}$-core partition, but $h^*_{2,3}(\la)=8$. Later, Wang and Yang \cite{WY} gave a characterization of $\overline{s}$-core partitions.

\begin{prop}\cite{WY}\label{prop:bar}
For a strict partition $\la=(\la_1,\la_2,\dots,\la_{\ell})$, $\la$ is an $\overline{s}$-core if and only if all the following hold:
\begin{enumerate}
\item[(a)] $s \notin \la$.
\item[(b)] If $\la_i \in \la$ with $\la_i>s$, then $\la_i -s \in \la$.
\item[(c)] If $\la_i, \la_j \in \la$, then $\la_i+\la_j \not\equiv 0 \pmod{s}$ except when $s$ is even and $\la_i,\la_j \equiv s/2 \pmod{s}$.
\end{enumerate}
\end{prop}

We extend this characterization to doubled distinct $s$-core partitions and $s$-CSYDs.

\begin{prop}\label{prop:dd}
For a strict partition $\la=(\la_1,\la_2,\dots,\la_{\ell})$, $\la\la$ is a doubled distinct $s$-core partition if and only if all the following hold:
\begin{enumerate}
\item[(a)] $\la$ is an $\overline{s}$-core.
\item[(b)] $s/2 \notin \la$ for even $s$.
\end{enumerate}
\end{prop}

\begin{proof}
It is known by Lemma \ref{lem:barhook} and the definition of $\la\la$ that $$\mathcal{H}(\la\la)=\mathcal{H}^*(\la) \cup \{h_{i,i}(\la\la)=2\la_i \mid i=1,2,\dots,\ell \}.$$ Therefore, for an $\overline{s}$-core partition $\la$ and even $s$, $s/2 \in \la$ if and only if $s \in \mathcal{H}(\la\la)$, meaning that $\la\la$ is not a doubled distinct $s$-core.
\end{proof}

\begin{prop}\label{prop:CSYD}
For a strict partition $\la=(\la_1,\la_2,\dots,\la_{\ell})$, $S(\la)$ is an $s$-CSYD if and only if all the following hold:
\begin{enumerate}
\item[(a)] $\la$ is an $\overline{s}$-core.
\item[(b)] $3s/2 \notin \la$ for even $s$.
\end{enumerate}
\end{prop}

\begin{proof}
 Assume first that $S(\la)$ is an $s$-CSYD. By Lemma \ref{lem:barhook}, $\la$ is an $\overline{s}$-core. If $3s/2 \in \la$, then $s/2 \in \la$ by Proposition \ref{prop:bar} (b). This implies that there is a bar length of $2s$ in $\la$, which means that $S(\la)$ is not an $s$-CSYD. 

 Conversely, suppose that two conditions (a) and (b) hold. If $\la$ is an $\overline{s}$-core but $S(\la)$ is not an $s$-CSYD, then there is a box $(i,j)$ in $S(\la)$ such that $h^*_{i,j}(\la)=sk$ for some $k\geq 2$. It follows from the definition of the bar lengths that there exist $\la_i,\la_j \in \la$ satisfying $\la_i+\la_j=sk$. Also, by Proposition~\ref{prop:bar}~(c), we deduce that $s$ is even and $\la_i,\la_j \equiv s/2 \pmod s$. Hence, when $\la_i > \la_j$, we can write $\la_i = (2m+1)s/2$ for some $m\geq 1$, and therefore $3s/2 \in \la$ by Proposition~\ref{prop:bar}~(b). It leads to a contradiction.
\end{proof}

\begin{rem} \label{rmk:oddoddodd}
From the characterizations we observe that, 
for coprime odd integers $s_1,s_2,\dots,s_p$, we have
\[
\mathcal{BC}_{(s_1, s_2, \dots, s_p)}=\mathcal{CS}_{(s_1, s_2, \dots, s_p)}=\mathcal{DD}_{(s_1, s_2, \dots, s_p)}.
\]
\end{rem}

\subsection{Generating functions}

In this subsection, we consider the generating functions of the following numbers, 
\begin{align*}
sc_s(n) &: \text{~the number of self-conjugate $s$-core partitions of $n$},\\ 
bc_s(n) &: \text{~the number of $\overline{s}$-core partitions of $n$},\\
cs_s(n) &: \text{~the number of  $s$-CSYDs of $n$},\\
dd_s(n) &: \text{~the number of doubled distinct $s$-core partitions of $n$}.  
\end{align*}

Garvan, Kim, and Stanton \cite{GKS} obtained the generating functions of the numbers $sc_s(n)$ and $dd_s(n)$ by using the concept of the core and the quotient of a partition. 

As usual, we use the well-known $q$-product notation $$(a;q)_n=\prod\limits_{i=0}^{n-1}(1-aq^i) \quad  \text{and} \quad (a;q)_{\infty}=\lim\limits_{n \to \infty} (a;q)_n \quad \text{for} ~ |q|<1.$$

\begin{prop}\cite[Equations (7.1a), (7.1b), (8.1a), and (8.1b)]{GKS}\label{prop:gf_GKS}
For a positive integer $s$, we have
\begin{align*}
\sum_{n=0}^{\infty}sc_s(n)q^n&=\begin{dcases*}
               \frac{(-q;q^2)_\infty(q^{2s};q^{2s})^{(s-1)/2}_\infty}{(-q^s;q^{2s})_\infty} & \text{if $s$ is odd},\\
               (-q;q^2)_\infty(q^{2s};q^{2s})^{s/2}_\infty    & \text{if $s$ is even,}
               \end{dcases*}\\[2ex]
\sum_{n=0}^{\infty}dd_s(n)q^n&=\begin{dcases*}
               \frac{(-q^2;q^2)_\infty(q^{2s};q^{2s})^{(s-1)/2}_\infty}{(-q^{2s};q^{2s})_\infty} & \text{if $s$ is odd},\\
               \frac{(-q^2;q^2)_\infty(q^{2s};q^{2s})^{(s-2)/2}_\infty}{(-q^{s};q^{s})_\infty}   & \text{if $s$ is even}.
               \end{dcases*}
\end{align*}
\end{prop}

The generating function of the numbers $bc_s(n)$ for odd $s$ was found by Olsson \cite{Olsson-book}. Note that for odd $s$, it is clear that $bc_s(n)=cs_s(n)$ as a partition $\la$ is an $\overline{s}$-core if and only if it is an $s$-CSYD by Propositions \ref{prop:bar} and \ref{prop:CSYD}.  

\begin{prop}\cite[Proposition (9.9)]{Olsson-book} \label{prop:gf_O}
For an odd integer $s$, we have
\[
\sum_{n=0}^{\infty}bc_{s}(n)q^n=\sum_{n=0}^{\infty}cs_{s}(n)q^n=\frac{(-q;q)_\infty(q^{s};q^{s})^{(s-1)/2}_\infty}{(-q^s;q^s)_\infty}.
\]
\end{prop}

From Propositions \ref{prop:gf_GKS} and \ref{prop:gf_O}, we also see that $dd_s(2n)=bc_{s}(n)$ when $s$ is odd.
We now give generating functions of the numbers $bc_{s}(n)$ and $cs_s(n)$ for even $s$ by using Propositions \ref{prop:bar}, \ref{prop:dd}, and \ref{prop:CSYD}. 

\begin{prop}\label{prop:bargen}
For an even integer $s$, we have
\[
\sum_{n=0}^{\infty}bc_{s}(n)q^n=\frac{(-q;q)_\infty(q^{s};q^{s})^{(s-2)/2}_\infty}{(-q^{s/2};q^{s/2})_\infty}\sum_{n\geq 0} q^{sn^2/2}.
\]
\end{prop}

\begin{proof}
Let $s$ be a fixed even integer. 
From Propositions \ref{prop:bar} and \ref{prop:dd} we first see that 
the number of $\overline{s}$-core partitions $\la$ of $n$ for which $s/2\notin \la$ is equal to $dd_s(2n)$. We also notice that for a positive integer $i$, the number of $\overline{s}$-core partitions $\la$ of $n$ for which $(2i-1)s/2\in \la$ and $(2i+1)s/2\notin \la$ is equal to $dd_s(2n-i^2s)$ since $(2i-1)s/2\in \la$ implies $(2i-3)s/2, (2i-5)s/2, \dots, s/2 \in \la$ by Proposition \ref{prop:bar} (b).
Therefore, we have 
\[
bc_s(n)=dd_s(2n)+dd_s(2n-s)+dd_s(2n-4s)+\cdots=\sum_{i\geq0} dd_s(2n-i^2s),
\]
which completes the proof from Proposition \ref{prop:gf_GKS}.
\end{proof}

\begin{prop}
For an even integer $s$, we have
\[
\sum_{n=0}^{\infty}cs_s(n)q^n=\frac{(-q;q)_\infty(q^{s};q^{s})^{(s-2)/2}_\infty}{(-q^s;q^{s/2})_\infty}.
\]
\end{prop}

\begin{proof}
Similar to the proof of Proposition \ref{prop:bargen}, $cs_s(n)=dd_s(2n)+dd_s(2n-s)$ for even $s$ by Propositions \ref{prop:dd} and \ref{prop:CSYD}.
\end{proof}



\section{Enumeration on $(s,t)$-cores} \label{sec:double}

A \emph{north-east ($NE$) lattice path} from $(0,0)$ to $(s,t)$ is a lattice path which consists of steps $N=(0,1)$ and $E=(1,0)$. Let $\mathcal{NE}(s,t)$ denote the set of all $NE$ lattice paths from $(0,0)$ to $(s,t)$.
In this section, we give $NE$ lattice path interpretations for $(\ols{s\phantom{t}},\overline{t})$-core related partitions and count such paths.

Combining the results on self-conjugate $(s,t)$-core partitions and $(\ols{s\phantom{t}},\overline{t})$-core partitions which are independently proved by Ford, Mai, and Sze \cite[Theorem 1]{FMS}, Bessenrodt and Olsson \cite[Theorem 3.2]{BO}, and Wang and Yang \cite[Theorem 1.3]{WY}, we get the following theorem.

\begin{thm}\cite{FMS,BO,WY}\label{thm:selfbar}
For coprime positive integers $s$ and $t$,
\[
|\mathcal{BC}_{(s,t)}|=|\mathcal{SC}_{(s,t)}|=\binom{\lfloor s/2 \rfloor + \lfloor t/2 \rfloor}{\lfloor s/2 \rfloor}.
\]
\end{thm}

Also, Ding \cite{Ding} examined the Hasse diagram of the poset structure of an $(s,s+1)$-CSYD to count them.

\begin{thm}\cite[Theorem 3.5]{Ding}\label{thm:Ding}
For any positive integer $s\geq 2$, 
\[
|\mathcal{CS}_{(s,s+1)}|=\binom{s-1}{\floor*{(s-1)/2}}+\binom{s-2}{\floor*{(s-1)/2}}.
\]
\end{thm}

From now on, we count doubled distinct $(s,t)$-cores 
and $(s,t)$-CSYDs. When $s$ and $t$ are both odd, the numbers of such partitions are already known by Remark \ref{rmk:oddoddodd}. We focus on the case when $s$ is even and $t$ is odd.

For $(\ols{s\phantom{t}},\overline{t})$-cores with coprime odd integers $s$ and $t$ such that $1<s<t$, Bessenrodt and Olsson \cite{BO} defined the Yin-Yang diagram as an array $A(s,t)=\{A_{i,j}\}$, where 

\[
A_{i,j}\coloneqq-\frac{s+1}{2}t+js+it \qquad \text{ for } 1 \le i \le \frac{s-1}{2} \text{ and } 1 \le j \le \frac{t-1}{2}.
\]

The location of $A_{i,j}$ is at the intersection of the $i$th row from the top and the $j$th column from the left. 
For fixed $s$ and $t$, they showed that 
the set of parts consisting of all possible $(\ols{s\phantom{t}},\overline{t})$-core partitions is equal to the set of absolute values of $A_{i,j}$ in $A(s,t)$.
They also gave a bijection $\phi$ between $\mathcal{BC}_{(s,t)}$ and the set $\mathcal{NE}((t-1)/2, (s-1)/2)$ in the Yin-Yang diagram from the lower-left corner to the upper-right corner. For an $NE$ lattice path $P$ in the Yin-Yang diagram $A(s,t)$, let $M(P)$ denote the set consisting of positive entries above $P$ and the absolute values of negative entries below $P$. According to the bijection $\phi$, if $\la$ is an $(\ols{s\phantom{t}},\overline{t})$-core partition and $P=\phi(\la)$ is the corresponding path in $A(s,t)$, then $M(P)$ is equal to the set of parts in $\la$.

For $(\ols{s\phantom{t}},\overline{t})$-cores with coprime even $s$ and odd $t$, Wang and Yang \cite{WY} defined the Yin-Yang diagram to be an array  $B(s,t)$, where
\[
B_{i,j}\coloneqq-\frac{s+2}{2}t+js+it \qquad \text{ for } 1 \le i \le \frac{s}{2} \text{ and } 1 \le j \le \frac{t-1}{2},
\]
and gave a bijection $\psi$ between the sets $\mathcal{BC}_{(s,t)}$ and  $\mathcal{NE}((t-1)/2, s/2)$ in $B(s,t)$ from the lower-left corner to the upper-right corner. Again, the map $\psi$ sends an $(\ols{s\phantom{t}},\overline{t})$-core $\la$ to the path $Q=\psi(\la)$ in $B(s,t)$, where $M(Q)$ is equal to the set of parts in $\la$. See Figure \ref{fig:YinYang} for example.

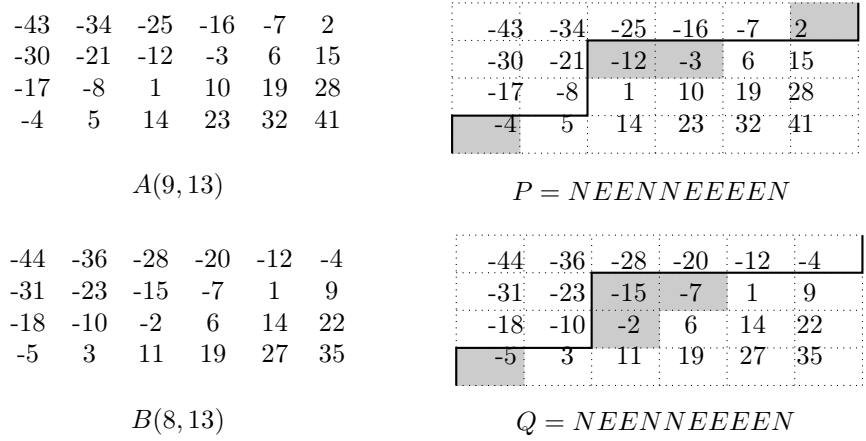
\begin{figure}[ht!]
\centering
\begin{tikzpicture}[scale=.5]
\node at (0,0){
\begin{tabular}{ c c c c c c }
-43 & -34 & -25 & -16 & -7 & 2\\
-30 & -21 & -12 & -3 & 6 & 15\\
-17 & -8 & 1 & 10 & 19 & 28\\
 -4 & 5 & 14 & 23 & 32 & 41
\end{tabular}};
\node at (0,-3) {$A(9,13)$};
\end{tikzpicture}
\qquad \quad
\begin{tikzpicture}[scale=.5]

\filldraw[color=gray!40]
(-5.3,-2) rectangle (-3.5, -1)
(-1.7,0) rectangle (1.9, 1)
(3.7,1) rectangle (5.5, 2)
;
\foreach \i in {0,1,2,3,4}
\draw[dotted] (-5.3,-2+\i)--(5.5,-2+\i);
\foreach \i in {0,1,2,3,4,5,6}
\draw[dotted] (-5.3+1.8*\i,-2)--(-5.3+1.8*\i,2);
\draw[thick] (-5.3,-2)--(-5.3,-1)--(-1.7,-1)--(-1.7,1)--(5.5,1)--(5.5,2);
\node at (0,0){
\begin{tabular}{ c c c c c c }
-43 & -34 & -25 & -16 & -7 & 2\\
-30 & -21 & -12 & -3 & 6 & 15\\
-17 & -8 & 1 & 10 & 19 & 28\\
-4 & 5 & 14 & 23 & 32 & 41
\end{tabular}};

\node at (0,-3) {$P=NEENNEEEEN$};
\end{tikzpicture}\\[2ex]

\begin{tikzpicture}[scale=.5]
\node at (0,0){
\begin{tabular}{ c c c c c c c}
 -44 & -36 & -28 & -20 & -12 & -4 \\ 
 -31 & -23 & -15 & -7 & 1 & 9 \\  
 -18 & -10 & -2 & 6 & 14 & 22\\
 -5 & 3 & 11 & 19 & 27 & 35
\end{tabular}};
\node at (0,-3) {$B(8,13)$};
\end{tikzpicture}
\qquad \quad
\begin{tikzpicture}[scale=.5]
\filldraw[color=gray!40]
(-5.3,-2) rectangle (-3.5, -1)
(-1.7,-1) rectangle (0.1,0)
(-1.7,0) rectangle (1.9, 1)
;
\foreach \i in {0,1,2,3,4}
\draw[dotted] (-5.3,-2+\i)--(5.5,-2+\i);
\foreach \i in {0,1,2,3,4,5,6}
\draw[dotted] (-5.3+1.8*\i,-2)--(-5.3+1.8*\i,2);
\draw[thick] (-5.3,-2)--(-5.3,-1)--(-1.7,-1)--(-1.7,1)--(5.5,1)--(5.5,2);
\node at (0,0){
\begin{tabular}{ c c c c c c c}
 -44 & -36 & -28 & -20 & -12 & -4 \\ 
 -31 & -23 & -15 & -7 & 1 & 9 \\  
 -18 & -10 & -2 & 6 & 14 & 22\\
 -5 & 3 & 11 & 19 & 27 & 35
\end{tabular}};
\node at (0,-3) {$Q=NEENNEEEEN$};
\end{tikzpicture}

\caption{The Yin-Yang diagrams $A(9,13)$ and $B(8,13)$, and the paths $P=\phi((12,4,3,2))$ and $Q=\psi((15,7,5,2))$.}\label{fig:YinYang}
\end{figure}

Now we give path interpretations for doubled distinct $(s,t)$-cores and $(s,t)$-CSYDs for even $s$ and odd $t$
by using this Yin-Yang diagram $B(s,t)$ together with Propositions~\ref{prop:dd} and \ref{prop:CSYD}. 

\begin{thm}\label{thm:dd2} 
 For even $s$ and odd $t$ that are coprime, there is a bijection between the sets  $\mathcal{DD}_{(s,t)}$ and  $\mathcal{NE}((t-1)/2,(s-2)/2)$. In addition,
\[
|\mathcal{DD}_{(s,t)}|=\binom{(s-2)/2 + (t-1)/2}{(s-2)/2}.
\]
\end{thm}

\begin{proof}
Recall the bijection $\psi$ between the sets $\mathcal{BC}_{(s,t)}$ and  $\mathcal{NE}((t-1)/2, s/2)$ in the Yin-Yang diagram $B(s,t)$ from the lower-left corner to the upper-right corner. To find the desired bijection, we restrict the domain of $\psi$ under the set $\mathcal{DD}_{(s,t)}$.
By Proposition~\ref{prop:dd}~(b) and the fact that $B_{1,(t-1)/2}=-s/2$, we see that $Q=\psi(\la)$ corresponds to a partition $\la$ such that $\la\la$ is a doubled distinct $(s,t)$-core if and only if $Q$ is a path in the set $\mathcal{NE}((t-1)/2, s/2)$ in the Yin-Yang diagram $B(s,t)$ that ends with a north step $N$, equivalently $\mathcal{NE}((t-1)/2, (s-2)/2)$.
Hence, the number of doubled distinct $(s,t)$-core partitions is given by $|\mathcal{NE}((t-1)/2, (s-2)/2)|$.
\end{proof}

\begin{thm}\label{thm:CSYD2}
For even $s$ and odd $t$ that are coprime, there is a bijection between the sets $\mathcal{CS}_{(s,t)}$ and
\[
\mathcal{NE}((t-1)/2,(s-2)/2)\cup \mathcal{NE}( (t-3)/2,(s-2)/2).
\]
In addition, 
\[
|\mathcal{CS}_{(s,t)}|=\binom{(s-2)/2 + (t-1)/2}{(s-2)/2}+\binom{(s-2)/2 + (t-3)/2}{(s-2)/2}.
\]
\end{thm}

\begin{proof}
It follows from Propositions~\ref{prop:bar} and \ref{prop:CSYD} that $\la$ is an $(s,t)$-CSYD if and only if $\la$ is an $(\ols{s\phantom{t}},\overline{t})$-core partitions and $3s/2 \notin \la$. 
We first note that $\la\la$ is a doubled distinct $(s,t)$-core partition if and only if $\la$ is an $(s,t)$-CSYD and $s/2 \notin \la$. Indeed, there is a bijection between the set of $(s,t)$-CSYDs $\la$ with $s/2 \notin \la$ and the set $\mathcal{NE}((t-1)/2, (s-2)/2)$ by Theorem~\ref{thm:dd2}. Therefore, it is sufficient to show that there is a bijection between the set of $(s,t)$-CSYDs $\la$ with $s/2 \in \la$ and the set $\mathcal{NE}((t-3)/2,(s-2)/2)$.

Note that for an $(s,t)$-CSYD $\la$ such that $s/2 \in \la$, $Q=\psi(\la)$ is a path in the set $\mathcal{NE}((t-1)/2, s/2)$ in the Yin-Yang diagram $B(s,t)$ that must end with an east step preceded by a north step since $B_{1,(t-1)/2}=-s/2$ and $B_{1,(t-3)/2}=-3s/2$.
Then, we get a bijection between the set of $(s,t)$-CSYDs $\la$ with $s/2 \in \la$ and the set  $\mathcal{NE}((t-3)/2,(s-2)/2)$. Moreover, the number of $(s,t)$-CSYDs is obtained by counting the corresponding lattice paths.
\end{proof}

\begin{proof}[Proof of Theorem \ref{thm:main1}]
It is followed by Remark \ref{rmk:oddoddodd}, Theorems \ref{thm:selfbar}, \ref{thm:dd2}, and \ref{thm:CSYD2}
\end{proof}


\section{Results on $(s,s+d,s+2d)$-cores}\label{sec:triple}

A path $P$ is called a \emph{free Motzkin path of type $(s,t)$} if it is a path from $(0,0)$ to $(s,t)$ which consists of steps $U=(1,1)$, $F=(1,0)$, and $D=(1,-1)$. Let $\mathcal{F}(s,t)$ be the set of free Motzkin paths of type $(s,t)$.
For given sets $A,B$ of sequences of steps, we denote
$\mathcal{F}(s,t \,;\, A,B)$ the set of free Motzkin paths $P$ of type $(s,t)$, where $P$ does not start with the sequences in the set $A$ and does not end with the sequences in the set $B$.

Recently, Cho and Huh \cite[Theorem 8]{ChoHuh} and Yan, Yan, and Zhou \cite[Theorems 1.1 and 1.2]{YYZ2} found a free Motzkin path interpretation of self-conjugate $(s,s+d,s+2d)$-core partitions and enumerated them independently.

\begin{thm}\cite{ChoHuh,YYZ2}
For coprime positive integers $s$ and $d$, there is a bijection between the sets $\mathcal{SC}_{(s,s+d,s+2d)}$ and
\begin{enumerate}
\item[(a)] $\mathcal{F}\left((s+d-1)/2,-d/2\right)$ if $s$ is odd and $d$ is even;
\item[(b)] $\mathcal{F}\left((s+d)/2,-(d+1)/2 \,;\, \emptyset,\{U\}\right)$ if $s$ is odd and $d$ is odd;
\item[(c)] $\mathcal{F}\left((s+d+1)/2,-(d+1)/2 \,;\, \emptyset,\{U\}\right)$ if $s$ is even and $d$ is odd.
\end{enumerate}
In addition, the number of self-conjugate $(s,s+d,s+2d)$-core partitions is
\[
\displaystyle
|\mathcal{SC}_{(s,s+d,s+2d)}|=
\begin{cases}
&\displaystyle\sum_{i=0}^{\lfloor s/4 \rfloor} \binom{(s+d-1)/2 }{i, d/2+i,  (s-1)/2-2i} \qquad \text{if $d$ is even,}\\
&\\
&\displaystyle\sum_{i=0}^{\lfloor s/2\rfloor} \binom{\lfloor (s+d-1)/2 \rfloor}{\lfloor i/2 \rfloor, \lfloor (d+i)/2\rfloor, \lfloor s/2 \rfloor -i} \quad \text{if $d$ is odd.}
\end{cases}
\]
\end{thm}


Similar to the construction in \cite{ChoHuh}, we give an abacus construction and a path interpretation for each set of $(\ols{s\phantom{d}},\overline{s+d},\overline{s+2d})$-core partitions, doubled distinct $(s,s+d,s+2d)$-core partitions, and $(s,s+d,s+2d)$-CSYDs.\\


\subsection{$(\ols{s\phantom{d}},\overline{s+d},\overline{s+2d})$-core partitions}\label{sec:bar}

For coprime positive integers $s$ and $d$, let the \emph{$(\overline{s+d},d)$-abacus diagram} be a diagram with infinitely many rows labeled by $i \in \mathbb{Z}$ and $\floor*{(s+d+2)/2}$ columns labeled by $j \in \{0,1,\dots,\floor*{(s+d)/2}\}$ from bottom to top and left to right whose position $(i,j)$ is labeled by $(s+d)i+dj$.

The following proposition guarantees that, for each positive integer $h$, there is at least one position on the $(\overline{s+d},d)$-abacus diagram labeled by either $h$ or $-h$.

\begin{prop} \label{prop:injection}
Let $s$ and $d$ be coprime positive integers and $h$ be a positive integer. For a given $(\overline{s+d},d)$-abacus diagram, we get the following properties.
\begin{itemize}
\item[(a)] If $h\not\equiv 0, (s+d)/2 \pmod{s+d}$, then there exists a unique position labeled by $h$ or $-h$.
\item[(b)] If $h\equiv 0 \pmod{s+d}$, then there are two positions labeled by $h$ and $-h$, respectively, in the first column. 
\item[(c)] If $s+d$ is even and $h\equiv (s+d)/2 \pmod{s+d}$, then there are two positions labeled by $h$ and $-h$, respectively, in the last column. 

\end{itemize} 
\end{prop}

\begin{proof}
In the $(\overline{s+d},d)$-abacus diagram, the absolute values of the labels in column $j$ are
congruent to $dj$ or $-dj$ modulo $s+d$.
We claim that $dj$ and $-dj$ for  $j\in\{0,1,\dots, \floor*{(s+d)/2}$\} are all incongruent modulo $s+d$ except $j=0$ or $(s+d)/2$. 
For $0 \leq j_1 < j_2\leq \floor*{(s+d)/2}$, it is clear that $dj_1$ and $dj_2$ are incongruent modulo $s+d$.
Suppose $dj_1 \equiv -dj_2 \pmod{s+d}$ for some $0 \leq j_1,j_2\leq \floor*{(s+d)/2}$, it follows that $d(j_1+j_2)$ is a multiple of $s+d$. Since $s$ and $d$ are coprime, $d(j_1+j_2)$ is not a multiple of $s+d$ except for $j_1=j_2=0$ or $j_1=j_2=(s+d)/2$, where both $s$ and $d$ are odd. This completes the proof of the claim. 
The claim implies that, for every positive integer $h$, there exists $j\in\{0,1,\dots, \floor*{(s+d)/2}\}$ such that $h$ is congruent to $dj$ or $-dj$ modulo $s+d$. In addition, if $h\not\equiv 0, (s+d)/2 \pmod{s+d}$, then there exists a unique position labeled by $h$ or $-h$ in the $(\overline{s+d},d)$-abacus diagram, which shows the statement (a). The statements (b) and (c) follows immediately. 
\end{proof}
 
For a strict partition $\la=(\la_1,\la_2,\dots)$, the \emph{$(\overline{s+d},d)$-abacus of $\la$} is obtained from the $(\overline{s+d},d)$-abacus diagram by placing a bead on position labeled by $\la_i$ if exists. Otherwise, we place a bead on position labeled by $-\la_i$. A position without bead is called a \emph{spacer}. See Figure \ref{fig:abacus_bar} for example.
We use this $(\overline{s+d},d)$-abacus when we deal with $(\ols{s\phantom{d}},\overline{s+d},\overline{s+2d})$-core partitions.
For the $(\overline{s+d},d)$-abacus of an $(\ols{s\phantom{d}},\overline{s+d},\overline{s+2d})$-core partition $\la$, let $r(j)$ denote the row number such that position $(r(j),j)$ is labeled by a positive integer while position $(r(j)-1,j)$ is labeled by a non-positive integer. The arrangement of beads on the diagram can be determined by the following rules. 

\begin{lem}\label{lem:beads}
Let $\la$ be a strict partition. For coprime positive integers $s$ and $d$, if $\la$ is an $(\ols{s\phantom{d}},\overline{s+d},\overline{s+2d})$-core, then the $(\overline{s+d},d)$-abacus of $\la$ satisfies the following.

\begin{enumerate}
\item[(a)] If a bead is placed on position $(i,j)$ such that $i> r(j)$, then a bead is also placed on each of positions $(i-1,j), (i-2,j), \dots, (r(j),j)$.
\item[(b)] If a bead is placed on position $(i,j)$ such that $i< r(j)-1$, then a bead is also placed on each of positions $(i+1,j), (i+2,j), \dots, (r(j)-1,j)$. 
\item[(c)] For each $j$, at most one bead is placed on positions $(r(j),j)$ or $(r(j)-1,j)$.
\end{enumerate}
\end{lem}

\begin{proof}
\begin{enumerate} 
\item[(a)] The fact that a bead is placed on position $(i,j)$ with $i>r(j)$ implies that $(s+d)i+dj$ is a part in $\la$. Since $\la$ is an $(\overline{s+d})$-core, it follows from Proposition~\ref{prop:bar}~(b) that $(s+d)(i-1)+dj$ is a part in $\la$. In a similar way, we also have $(s+d)(i-2)+dj, \dots, (s+d)r(j)+dj \in \la$ so that a bead is placed on each of positions $(i-1,j), (i-2,j), \dots, (r(j),j)$.
\item[(b)] If a bead is placed on position $(i,j)$ with $i<r(j)-1$, then $-(s+d)i-dj$ is a part in $\la$. Again, it follows from Proposition~\ref{prop:bar}~(b) that $-(s+d)(i+1)-dj$ is a part in $\la$ and so are  $-(s+d)(i+2)-dj, \dots, -(s+d)(r(j)-1)-dj \in \la$. Thus, we place a bead on  positions $(i+1,j), (i+2,j), \dots, (r(j)-1,j)$.
\item[(c)] Suppose that beads are placed on both positions $(r(j),j)$ and $(r(j)-1,j)$ labeled by $(s+d)r(j)+dj$ and $(s+d)(r(j)-1)+dj$, respectively.
One can notice that $(s+d)(r(j)-1)+dj$ is a non-positive integer and the sum of the absolute values of $(s+d)r(j)+dj$ and $(s+d)(r(j)-1)+dj$ is $s+d$, which contradicts to Proposition~\ref{prop:bar}~(c). In particular, if one of them is labeled by $(s+d)/2$, then the other must be labeled by $-(s+d)/2$, which is also a contradiction to the definition of the $(\overline{s+d},d)$-abacus. 
\end{enumerate}
\end{proof}

For an $(\ols{s\phantom{d}},\overline{s+d},\overline{s+2d})$-core partition $\la$, in order to explain the properties of the $(\overline{s+d},d)$-abacus of $\la$ more simply, we define the \emph{$(\overline{s+d},d)$-abacus function of $\la$}
\[
f:\{0,1,\dots,\lfloor (s+d)/2 \rfloor\}\rightarrow \mathbb{Z}
\]
as follows:
For each $j \in \{0,1,\dots,\lfloor (s+d)/2 \rfloor\}$, 
if there is a bead labeled by a positive integer in column $j$, let $f(j)$ be the largest row number in column $j$, where a bead is placed on. Otherwise, let $f(j)$ be the largest row number in column $j$, where position $(f(j),j)$ is a spacer with a non-positive labeled number. 

The following propositions give some basic properties of the $(\overline{s+d},d)$-abacus function of an $(\ols{s\phantom{d}},\overline{s+d},\overline{s+2d})$-core partition. 

\begin{prop}\label{prop:f_initial}
Let $s$ and $d$ be coprime positive integers. If $\la$ is an $(\ols{s\phantom{d}},\overline{s+d},\overline{s+2d})$-core partition, then the $(\overline{s+d},d)$-abacus function $f$ of $\la$ satisfies the following.
\begin{enumerate}
\item[(a)] $f(0)=0$ and $f(1)=0$ or $-1$.
\item[(b)] $f(j-1)$ is equal to one of the three values $f(j)-1$, $f(j)$, and $f(j)+1$ for $j=1,2,\dots, \lfloor(s+d)/2\rfloor$.
\end{enumerate}
\end{prop}

\begin{proof}
We consider the $(\overline{s+d},d)$-abacus of $\la$. 
\begin{enumerate}
\item[(a)] Since positions $(0,0)$ and $(1,0)$ are  
labeled by $0$ and $s+d$, respectively, there is no bead in column $0$. Hence, $f(0)=0$. Similarly, since positions $(-1,1)$, $(0,1)$, and $(1,1)$ are labeled by $-s$, $d$, and $s+2d$ respectively, there is at most one bead on position $(0,1)$ in column $1$. Hence, $f(1)=0$ or $-1$.
\item[(b)] For a fixed $j$, let $f(j)=i$.
Suppose that a bead is placed on position $(i,j)$ which is labeled by a positive integer. If position $(i-1,j-1)$ is labeled by a positive integer, then a bead is placed on this position by Proposition~\ref{prop:bar}~(b). Otherwise, position $(i-1,j-1)$ is a spacer by Proposition~\ref{prop:bar}~(c). In any case, it follows from the definition of $f$ that $f(j-1)\geq f(j)-1$. Additionally, since position $(i+1,j)$ is a spacer, position $(i+2,j-1)$ is a spacer by Proposition~\ref{prop:bar}~(b). Hence, $f(j-1)\leq f(j)+1$.

Next, suppose that position $(i,j)$ is a spacer which is labeled by a negative integer. Since position $(i-1,j-1)$ is labeled by a negative integer, it is a spacer, so $f(j-1)\geq f(j)-1$. We now assume that $f(j-1)\geq i+2$. If position $(i+2,j-1)$ is labeled by a positive integer, then a bead is placed on this position by Lemma~\ref{lem:beads}~(a). In this case, position $(i+1,j)$ either has with a bead labeled by a positive integer or is a spacer labeled by a negative integer by Proposition~\ref{prop:bar}~(b) and (c), which contradicts to $f(j)=i$. Otherwise, if position $(i+2,j-1)$ is labeled by a negative integer, then it is a spacer. Therefore, position $(i+1,j)$ is a spacer by Proposition~\ref{prop:bar}~(b), which also contradicts to $f(j)=i$. Hence, $f(j-1)\leq f(j)+1$.

\end{enumerate}
\end{proof}

\begin{prop}\label{prop:barf}
Let $s$ and $d$ be coprime integers. For an $(\ols{s\phantom{d}},\overline{s+d},\overline{s+2d})$-core partition $\la$, the $(\overline{s+d},d)$-abacus function $f$ of $\la$ satisfies the following.
\begin{enumerate}
    \item [(a)] If $s$ is odd and $d$ is even, then $f(\frac{s+d-1}{2})\in \{-\frac{d+2}{2}, -\frac{d}{2}\}$.
    \item [(b)] If $s$ and $d$ are both odd, then $f(\frac{s+d}{2}) \in  \{-\frac{d+1}{2},-\frac{d-1}{2}\}$. In addition, $f(\frac{s+d-2}{2})=-\frac{d+1}{2}$ when $f(\frac{s+d}{2})=-\frac{d-1}{2}$.
    \item [(c)] If $s$ is even and $d$ is odd, then $f(\frac{s+d-1}{2})\in \{-\frac{d+3}{2}, -\frac{d+1}{2},  -\frac{d-1}{2}\}$.
\end{enumerate}
\end{prop}

\begin{proof}
Let position $(a,b)$ denote position $(-\lfloor d/2 \rfloor,\lfloor (s+d)/2 \rfloor)$.
\begin{enumerate}
    \item [(a)] Positions $(a-1,b),(a,b)$, and $(a+1,b)$ are labeled by $-s-3d/2, -d/2$, and
    $s+d/2$, respectively.
    First we show that $s+d/2$ and $s+3d/2$ are not parts of $\la$. If $s+d/2 \in \la$, then $d/2\in\la$ by Proposition \ref{prop:bar} (b). It gives a contradiction by Proposition \ref{prop:bar} (c) since $(s+d/2)+d/2=s+d$. One can similarly show that $s+3d/2 \notin \la$. Hence, the only possibility of having a bead in column $b$ is putting it on position $(a,b)$. Thus, $f(b)=a-1$ or $a$. 
    \item [(b)] Positions
    $(a-1,b),(a,b)$, and $(a+1,b)$ are labeled by $-(s+d)/2$, $(s+d)/2$, and $(3s+3d)/2$, respectively.
    We first claim that there is no bead on position $(a+1,b)$. If $(3s+3d)/2 \in \la$, then $(s+d)/2,(s+3d)/2 \in \la$ by Proposition \ref{prop:bar} (b), which contradicts to Proposition \ref{prop:bar} (c) since $(s+d)/2 + (s+3d)/2 = s+2d$. This completes a proof of the claim.
    Therefore, $f(b)=a$ when $(s+d)/2 \in \la$ and $f(b)=a-1$ otherwise. 
    
    Furthermore, we would like to show that $f(b-1)=a-1$ assuming that $f(b)=a$. Consider positions $(a-1,b-1)$ and $(a,b-1)$ which are labeled by $-(s+3d)/2$ and $(s-d)/2$, respectively. Position $(a-1,b-1)$ is a spacer by Proposition \ref{prop:bar} (c) since $(s+3d)/2+(s+d)/2=s+2d$.    
    When $s>d$, position $(a,b-1)$ is also a spacer by Proposition \ref{prop:bar} (c) since $(s-d)/2+(s+d)/2=s$. Otherwise, $(s-d)/2$ is negative and a bead is placed on position $(a,b-1)$ since $(d-s)/2=(s+d)/2-s$.
    In any case, we conclude that $f(b-1)=a-1$.
    \item [(c)] Positions $(a-2,b), (a-1,b), (a,b)$, and $(a+1,b)$ are labeled by $-(3s+4d)/2,-(s+2d)/2,s/2$ and $(3s+2d)/2$, respectively.
    If $(3s+2d)/2 \in \la$, then $s/2, (s+2d)/2\in\la$ by Proposition \ref{prop:bar} (b), which contradicts to Proposition \ref{prop:bar} (c). Thus, $(3s+2d)/2 \notin \la$ and $f(b)<a+1$. Similarly, $(3s+4d)/2 \notin \la$ which implies $f(b)\geq a-2$.
\end{enumerate}
\end{proof}

For coprime positive integers $s$ and $d$, it is obvious that the map from the set of $(\ols{s\phantom{d}}, \overline{s+d}, \overline{s+2d})$-core partitions to the set of functions satisfying the conditions in Propositions \ref{prop:f_initial} and \ref{prop:barf} is well-defined and injective. The following proposition shows that this map is surjective.

\begin{prop}\label{prop:barinv}
For coprime positive integers $s$ and $d$, let $f$ be a function that satisfies Propositions \ref{prop:f_initial} and \ref{prop:barf}. If $\la$ is a strict partition such that $f$ is the $(\overline{s+d},d)$-abacus function of $\la$, then $\la$ is an $(\ols{s\phantom{d}},\overline{s+d},\overline{s+2d})$-core partition.
\end{prop}

\begin{proof}
We show that $\la$ satisfies the conditions in Proposition \ref{prop:bar} (a), (b), and (c).
\begin{enumerate}
    \item [(a)] It follows from Proposition \ref{prop:f_initial} (a) that $s,s+d,s+2d \notin \la$.
    \item [(b)] Assume that $h$ is a part in $\la$. If $h > s+d$, then $h - (s+d) \in \la$ by Lemma \ref{lem:beads}.
    Consider the $(\overline{s+d},d)$-abacus diagram and suppose that $h > s$, but $h - s \notin \la$ to the contrary. Let position $(i,j)$ be labeled by $a$ such that $|a|=h$ which has a bead on. If $a>0$, then we get $j<\floor*{(s+d)/2}$ or $h=(s+d)/2$ with $s<d$ for odd numbers $s$ and $d$ by Proposition \ref{prop:barf}.
    First, assume that $j<\floor*{(s+d)/2}$. Then, position $(i-1,j+1)$ is a spacer labeled by $h -s$ which implies $f(j)\geq i$ and $f(j+1)<i-1$, so we get a contradiction to Proposition \ref{prop:f_initial} (b). Now, for odd numbers $s$ and $d$, let $h=(s+d)/2$ with $s<d$. Then, we have a bead on position $(-(d-1)/2,(s+d-2)/2)$ labeled by $(s-d)/2$ by Proposition \ref{prop:barf} (b), which gives a contradiction.
    If $a<0$, then position $(i+1,j-1)$ labeled by $-h +s$ is a spacer. This implies that $f(j-1) \geq i+1$ and $f(j) < i$, which contradicts to Proposition \ref{prop:f_initial} (b).
    By the similar argument, one can show that $h > s+2d$ implies $h - (s+2d) \in \la$.
    \item [(c)] By Lemma \ref{lem:beads} (c) and the construction of $f$, it is sufficient to show that there are no $h_1,h_2 \in \la$ such that $h_1 \neq h_2$ and $h_1 + h_2 \in \{s,s+2d\}$. Assume that there exist $h_1,h_2 \in \la$ satisfying $h_1 + h_2 =s$. If $h_1, h_2\neq (s+d)/2$, then there are positions $(i,j)$ and $(i-1,j+1)$ that are labeled by $h_1$ and $-h_2$, respectively. In this case, we get $f(j)\geq i$ and $f(j+1) < i-1$, which contradicts to Proposition \ref{prop:f_initial} (b). 
    If $h_2=(s+d)/2$ (so both $s$ and $d$ are odd), then positions $(i,(s+d-2)/2)$ and $(i,(s+d)/2)$ are labeled by $h_1$ and $h_2$, respectively, and we get a contradiction to Proposition \ref{prop:barf} (b). Similar argument works for the case when $h_1+h_2 \neq s+2d$. 
\end{enumerate}
\end{proof}

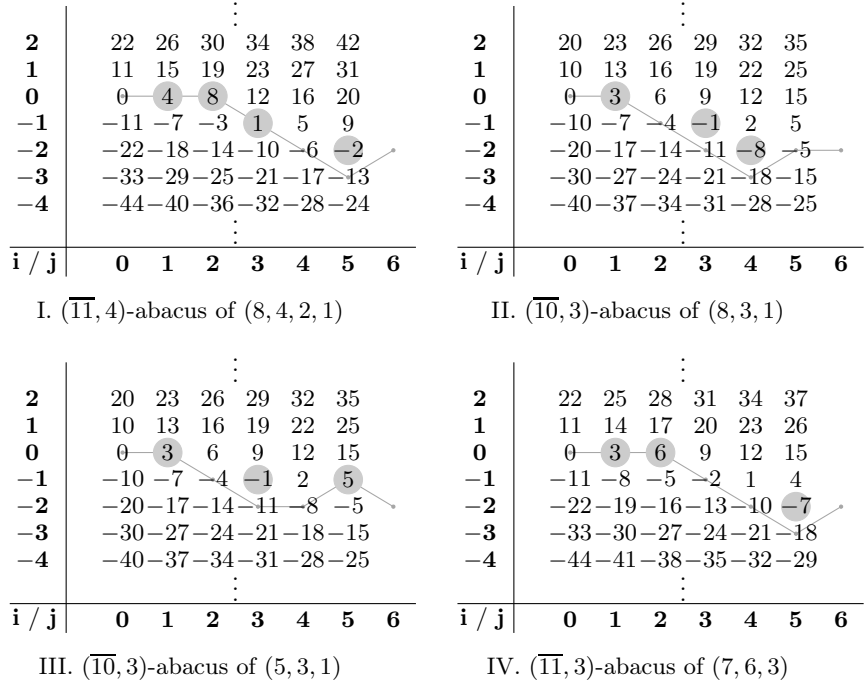
\begin{figure}[ht!]
\centering
\begin{tikzpicture}[scale=.3]
\small
\draw[color=gray!70] (0,0)--(2,0)--(4,0)--(6,-1.2)--(8,-2.4)--(10,-3.6)--(12,-2.4);

\filldraw[color=gray!70]
(0,0) circle (2pt)
(2,0) circle (2pt)
(4,0) circle (2pt)
(6,-1.2) circle (2pt)
(8,-2.4) circle (2pt)
(10,-3.6) circle (2pt)
(12,-2.4) circle (2pt)
;

\filldraw[color=gray!40] (2,0) circle (18pt);
\filldraw[color=gray!40] (4,0) circle (18pt);
\filldraw[color=gray!40] (6,-1.2) circle (18pt);
\filldraw[color=gray!40] (10,-2.4) circle (18pt);

\node at (-4,2.4) {$\mathbf{2}$};
\node at (-4,1.2) {$\mathbf{1}$};
\node at (-4,0) {$\mathbf{0}$};
\node at (-4,-1.2) {$\mathbf{-1}$};
\node at (-4,-2.4) {$\mathbf{-2}$};
\node at (-4,-3.6) {$\mathbf{-3}$};
\node at (-4,-4.8) {$\mathbf{-4}$};

\node at (-3.9,-7.4) {$\mathbf{i~/~j}$};

\node at (0,-7.4) {$\mathbf{0}$};
\node at (2,-7.4) {$\mathbf{1}$};
\node at (4,-7.4) {$\mathbf{2}$};
\node at (6,-7.4) {$\mathbf{3}$};
\node at (8,-7.4) {$\mathbf{4}$};
\node at (10,-7.4) {$\mathbf{5}$};
\node at (12,-7.4) {$\mathbf{6}$};

\foreach \i in {22,26,30,34,38,42}
\node at (\i/2-22/2,2.4) {$\i$};
\foreach \i in {11,15,19,23,27,31}
\node at (\i/2-11/2,1.2) {$\i$};
\foreach \i in {0,4,8,12,16,20}
\node at (\i/2,0) {$\i$};
\foreach \i in {-11,-7,-3,1,5,9}
\node at (\i/2+11/2,-1.2) {$\i$};
\foreach \i in {-22,-18,-14,-10,-6,-2}
\node at (\i/2+22/2,-2.4) {$\i$};
\foreach \i in {-33,-29,-25,-21,-17,-13}
\node at (\i/2+33/2,-3.6) {$\i$};
\foreach \i in {-44,-40,-36,-32,-28,-24}
\node at (\i/2+44/2,-4.8) {$\i$};

\node at (5,4) {\vdots};
\node at (5,-5.7) {\vdots};

\draw (-2.5,4)--(-2.5,-8);

\draw (-5,-6.7)--(13,-6.7);

\node at (3,-9.5) {I. $(\overline{11},4)$-abacus of $(8,4,2,1)$};
\end{tikzpicture}
\quad
\begin{tikzpicture}[scale=.3]
\small
\draw[color=gray!70] (0,0)--(2,0)--(4,-1.2)--(6,-2.4)--(8,-3.6)--(10,-2.4)--(12,-2.4);

\filldraw[color=gray!70]
(0,0) circle (2pt)
(2,0) circle (2pt)
(4,-1.2) circle (2pt)
(6,-2.4) circle (2pt)
(8,-3.6) circle (2pt)
(10,-2.4) circle (2pt)
(12,-2.4) circle (2pt)
;

\filldraw[color=gray!40] (2,0) circle (18pt);
\filldraw[color=gray!40] (6,-1.2) circle (18pt);
\filldraw[color=gray!40] (8,-2.4) circle (18pt);

\node at (-4,2.4) {$\mathbf{2}$};
\node at (-4,1.2) {$\mathbf{1}$};
\node at (-4,0) {$\mathbf{0}$};
\node at (-4,-1.2) {$\mathbf{-1}$};
\node at (-4,-2.4) {$\mathbf{-2}$};
\node at (-4,-3.6) {$\mathbf{-3}$};
\node at (-4,-4.8) {$\mathbf{-4}$};

\node at (-3.9,-7.4) {$\mathbf{i~/~j}$};

\node at (0,-7.4) {$\mathbf{0}$};
\node at (2,-7.4) {$\mathbf{1}$};
\node at (4,-7.4) {$\mathbf{2}$};
\node at (6,-7.4) {$\mathbf{3}$};
\node at (8,-7.4) {$\mathbf{4}$};
\node at (10,-7.4) {$\mathbf{5}$};
\node at (12,-7.4) {$\mathbf{6}$};

\foreach \i in {20,23,26,29,32,35}
\node at (\i/1.5-20/1.5,2.4) {$\i$};
\foreach \i in {10,13,16,19,22,25}
\node at (\i/1.5-10/1.5,1.2) {$\i$};
\foreach \i in {0,3,6,9,12,15}
\node at (\i/1.5,0) {$\i$};
\foreach \i in {-10,-7,-4,-1,2,5}
\node at (\i/1.5+10/1.5,-1.2) {$\i$};
\foreach \i in {-20,-17,-14,-11,-8,-5}
\node at (\i/1.5+20/1.5,-2.4) {$\i$};
\foreach \i in {-30,-27,-24,-21,-18,-15}
\node at (\i/1.5+30/1.5,-3.6) {$\i$};
\foreach \i in {-40,-37,-34,-31,-28,-25}
\node at (\i/1.5+40/1.5,-4.8) {$\i$};

\node at (5,4) {\vdots};
\node at (5,-5.7) {\vdots};

\draw (-2.5,4)--(-2.5,-8);

\draw (-5,-6.7)--(13,-6.7);

\node at (3,-9.5) {II. $(\overline{10},3)$-abacus of $(8,3,1)$};
\end{tikzpicture}\\

\begin{tikzpicture}[scale=.3]
\small
\draw[color=gray!70] (0,0)--(2,0)--(4,-1.2)--(6,-2.4)--(8,-2.4)--(10,-1.2)--(12,-2.4);

\filldraw[color=gray!70]
(0,0) circle (2pt)
(2,0) circle (2pt)
(4,-1.2) circle (2pt)
(6,-2.4) circle (2pt)
(8,-2.4) circle (2pt)
(10,-1.2) circle (2pt)
(12,-2.4) circle (2pt)
;

\filldraw[color=gray!40] (2,0) circle (18pt);
\filldraw[color=gray!40] (6,-1.2) circle (18pt);
\filldraw[color=gray!40] (10,-1.2) circle (18pt);

\node at (-4,2.4) {$\mathbf{2}$};
\node at (-4,1.2) {$\mathbf{1}$};
\node at (-4,0) {$\mathbf{0}$};
\node at (-4,-1.2) {$\mathbf{-1}$};
\node at (-4,-2.4) {$\mathbf{-2}$};
\node at (-4,-3.6) {$\mathbf{-3}$};
\node at (-4,-4.8) {$\mathbf{-4}$};

\node at (-3.9,-7.4) {$\mathbf{i~/~j}$};

\node at (0,-7.4) {$\mathbf{0}$};
\node at (2,-7.4) {$\mathbf{1}$};
\node at (4,-7.4) {$\mathbf{2}$};
\node at (6,-7.4) {$\mathbf{3}$};
\node at (8,-7.4) {$\mathbf{4}$};
\node at (10,-7.4) {$\mathbf{5}$};
\node at (12,-7.4) {$\mathbf{6}$};

\foreach \i in {20,23,26,29,32,35}
\node at (\i/1.5-20/1.5,2.4) {$\i$};
\foreach \i in {10,13,16,19,22,25}
\node at (\i/1.5-10/1.5,1.2) {$\i$};
\foreach \i in {0,3,6,9,12,15}
\node at (\i/1.5,0) {$\i$};
\foreach \i in {-10,-7,-4,-1,2,5}
\node at (\i/1.5+10/1.5,-1.2) {$\i$};
\foreach \i in {-20,-17,-14,-11,-8,-5}
\node at (\i/1.5+20/1.5,-2.4) {$\i$};
\foreach \i in {-30,-27,-24,-21,-18,-15}
\node at (\i/1.5+30/1.5,-3.6) {$\i$};
\foreach \i in {-40,-37,-34,-31,-28,-25}
\node at (\i/1.5+40/1.5,-4.8) {$\i$};

\node at (5,4) {\vdots};
\node at (5,-5.7) {\vdots};

\draw (-2.5,4)--(-2.5,-8);

\draw (-5,-6.7)--(13,-6.7);

\node at (3,-9.5) {III. $(\overline{10},3)$-abacus of $(5,3,1)$};
\end{tikzpicture}
\quad
\begin{tikzpicture}[scale=.3]
\small
\draw[color=gray!70] (0,0)--(2,0)--(4,0)--(6,-1.2)--(8,-2.4)--(10,-3.6)--(12,-2.4);

\filldraw[color=gray!70]
(0,0) circle (2pt)
(2,0) circle (2pt)
(4,0) circle (2pt)
(6,-1.2) circle (2pt)
(8,-2.4) circle (2pt)
(10,-3.6) circle (2pt)
(12,-2.4) circle (2pt)
;

\filldraw[color=gray!40] (2,0) circle (18pt);
\filldraw[color=gray!40] (4,0) circle (18pt);
\filldraw[color=gray!40] (10,-2.4) circle (18pt);

\node at (-4,2.4) {$\mathbf{2}$};
\node at (-4,1.2) {$\mathbf{1}$};
\node at (-4,0) {$\mathbf{0}$};
\node at (-4,-1.2) {$\mathbf{-1}$};
\node at (-4,-2.4) {$\mathbf{-2}$};
\node at (-4,-3.6) {$\mathbf{-3}$};
\node at (-4,-4.8) {$\mathbf{-4}$};

\node at (-3.9,-7.4) {$\mathbf{i~/~j}$};

\node at (0,-7.4) {$\mathbf{0}$};
\node at (2,-7.4) {$\mathbf{1}$};
\node at (4,-7.4) {$\mathbf{2}$};
\node at (6,-7.4) {$\mathbf{3}$};
\node at (8,-7.4) {$\mathbf{4}$};
\node at (10,-7.4) {$\mathbf{5}$};
\node at (12,-7.4) {$\mathbf{6}$};

\foreach \i in {22,25,28,31,34,37}
\node at (\i/1.5-22/1.5,2.4) {$\i$};
\foreach \i in {11,14,17,20,23,26}
\node at (\i/1.5-11/1.5,1.2) {$\i$};
\foreach \i in {0,3,6,9,12,15}
\node at (\i/1.5,0) {$\i$};
\foreach \i in {-11,-8,-5,-2,1,4}
\node at (\i/1.5+11/1.5,-1.2) {$\i$};
\foreach \i in {-22,-19,-16,-13,-10,-7}
\node at (\i/1.5+22/1.5,-2.4) {$\i$};
\foreach \i in {-33,-30,-27,-24,-21,-18}
\node at (\i/1.5+33/1.5,-3.6) {$\i$};
\foreach \i in {-44,-41,-38,-35,-32,-29}
\node at (\i/1.5+44/1.5,-4.8) {$\i$};

\node at (5,4) {\vdots};
\node at (5,-5.7) {\vdots};

\draw (-2.5,4)--(-2.5,-8);

\draw (-5,-6.7)--(13,-6.7);

\node at (3,-9.5) {IV. $(\overline{11},3)$-abacus of $(7,6,3)$};
\end{tikzpicture}
\caption{The $(\overline{s+d},d)$-abaci of several partitions and the corresponding free Motzkin paths}\label{fig:abacus_bar}
\end{figure}

For given coprime integers $s$ and $d$, let $\la$ be an $(\ols{s\phantom{d}}, \overline{s+d}, \overline{s+2d})$-core partition. For the $(\overline{s+d},d)$-abacus function $f$ of $\la$, we set $f(\floor*{(s+d+2)/2})\coloneqq -\floor*{(d+1)/2}$ and 
define $\phi(\la)$ to be the path $P=P_1P_2 \cdots P_{\floor*{(s+d+2)/2}}$, where the $j$th step is given by $P_j=(1,f(j)-f(j-1))$ for each $j$. By Proposition \ref{prop:f_initial} (b), $P_j$ is one of the three steps $U=(1,1)$, $F=(1,0)$, and $D=(1,-1)$, so $P$ is a free Motzkin path. From this construction together with Proposition~\ref{prop:barf}, we obtain a path interpretation of an $(\ols{s\phantom{d}}, \overline{s+d}, \overline{s+2d})$-core partition as described in the following theorem.

\begin{thm}\label{thm:barcore}
For coprime positive integers $s$ and $d$, there is a bijection between the sets $\mathcal{BC}_{(s,s+d,s+2d)}$ and 
\begin{enumerate}
\item[(a)] 
$\mathcal{F}(\frac{s+d+1}{2},-\frac{d}{2} \,;\, \{U\},\{D\})$ if $s$ is odd and $d$ is even;
\item[(b)] $\mathcal{F}(\frac{s+d+2}{2},-\frac{d+1}{2} \,;\, \{U\},\{FD,DD,U\})$ if both $s$ and $d$ are odd;
\item[(c)] $\mathcal{F}(\frac{s+d+1}{2},-\frac{d+1}{2} \,;\,
\{U\},\emptyset)$ if $s$ is even and $d$ is odd.
\end{enumerate}
\end{thm}

\begin{proof} 
All the bijections come from Propositions \ref{prop:f_initial}, \ref{prop:barf}, and \ref{prop:barinv}. By drawing line segments that connects the positions $(f(j),j)$ and $(f(j+1),j+1)$ to obtain $P=P_1P_2 \cdots P_{\floor*{(s+d)/2}}$ in the $(\overline{s+d},d)$-abacus, we have the one-to-one correspondences between the sets $\mathcal{BC}_{(s,s+d,s+2d)}$ and

{
\small
\begin{align*}
\text{(a) }& \mathcal{F}\left(\frac{s+d-1}{2},-\frac{d}{2}\,;\, \{U\},\emptyset\right)\cup\mathcal{F}\left(\frac{s+d-1}{2}, -\frac{d+2}{2}\,;\, \{U\},\emptyset\right);\\
\text{(b) }& \mathcal{F}\left(\frac{s+d}{2},-\frac{d+1}{2} \,;\, \{U\},\emptyset\right) \cup \mathcal{F}\left(\frac{s+d}{2},-\frac{d-1}{2} \,;\, \{U\},\{F,D\}\right);\\
\text{(c) }& \mathcal{F}\left(\frac{s+d-1}{2},-\frac{d-1}{2} \,;\, \{U\},\emptyset\right) \cup \mathcal{F}\left(\frac{s+d-1}{2},-\frac{d+1}{2} \,;\, \{U\},\emptyset\right) \\ &\hspace{55mm}\cup\mathcal{F}\left(\frac{s+d-1}{2},-\frac{d+3}{2} \,;\, \{U\},\emptyset\right).
\end{align*}
}

The addition of the last step gives free Motzkin paths of type $(\lfloor (s+d+2)/2\rfloor,-\lfloor (d+1)/2 \rfloor)$ as we desired.

\end{proof}

\begin{ex}
For a $(\overline{7}, \overline{11}, \overline{15})$-core partition $\la=(8,4,2,1)$, Diagram I in Figure \ref{fig:abacus_bar} illustrates the $(\overline{11},4)$-abacus of $\la$. The $(\overline{11},4)$-abacus function $f$ of $\la$ is given by
$$f(0)=0,~ f(1)=0,~ f(2)=0,~ f(3)=-1,~ f(4)=-2, ~f(5)=-3, ~f(6)=-2,$$
and its corresponding path is $P=\phi(\la)=FFDDDU$.
\end{ex}


\subsection{Doubled distinct $(s,s+d,s+2d)$-core partitions}

Recall that for an $\overline{s}$-core partition $\la$ with even $s$, $\la\la$ is a doubled distinct $s$-core if and only if $s/2 \notin \la$.

\begin{prop}\label{prop:dd_f}
For a strict partition $\la$ such that $\la\la$ is a doubled distinct $(s,s+d,s+2d)$-core, the $(\overline{s+d},d)$-abacus function $f$ of $\la$ satisfies the following.
\begin{enumerate}
    \item [(a)] If $s$ is odd and $d$ is even, then $f(\frac{s+d-1}{2})\in \{ -\frac{d+2}{2}, -\frac{d}{2}\}$.
    \item [(b)] If $s$ and $d$ are both odd, then $f(\frac{s+d}{2})=-\frac{d+1}{2}$.
    \item [(c)] If $s$ is even and $d$ is odd, then $f(\frac{s+d-1}{2})=-\frac{d+1}{2}$.
\end{enumerate}
\end{prop}

\begin{proof}
\begin{enumerate}
    \item [(a)] It follows from Proposition \ref{prop:barf} (a) since we do not need to consider the additional property of a doubled distinct core partition. 
    \item [(b)] Positions $(-(d+1)/2,(s+d)/2)$ and $(-(d-1)/2,(s+d)/2)$  are labeled by $-(s+d)/2$ and  $(s+d)/2$, respectively. Since $(s+d)/2 \notin \la$ by Proposition \ref{prop:dd} (b), there is no bead in column $(s+d)/2$, and $f((s+d)/2)=-(d+1)/2$.
    \item [(c)] Positions $(-(d+1)/2,(s+d-1)/2)$ and $(-(d-1)/2,(s+d-1)/2)$ are labeled by $-(s+2d)/2$ and $s/2$, respectively. We know that $s/2,(s+2d)/2 \notin \la$ by Proposition \ref{prop:dd} (b), so $f((s+d-1)/2)=-(d+1)/2$.
\end{enumerate}
\end{proof}

Similar to the bar-core case considered in Section \ref{sec:bar}, there is a one-to-one correspondence between the set of doubled distinct $(s,s+d,s+2d)$-cores and the set of functions satisfying the conditions in Propositions \ref{prop:f_initial} and \ref{prop:dd_f}. The following proposition completes the existence of the bijection.

\begin{prop}\label{prop:dd_inverse}
For coprime positive integers $s$ and $d$, let $f$ be a function that satisfies Propositions \ref{prop:f_initial} and \ref{prop:dd_f}. If $\la$ is a strict partition such that $f$ is the $(\overline{s+d},d)$-abacus function of $\la$, then $\la\la$ is a doubled distinct $(s,s+d,s+2d)$-core.
\end{prop}

\begin{proof}
It is sufficient to show that $\la$ satisfies Proposition \ref{prop:dd} (b). We consider the case according to the parity of $s$ and $d$.
For odd $s$ and even $d$, all of $s,s+d,s+2d$ are odd, so we no longer need to consider the additional property of 
$\la\la$.
For odds $s$ and $d$, there is no bead in column $(s+d)/2$ by Proposition \ref{prop:dd_f} (b). Since the only column that has labels whose absolute values are $(s+d)/2$ is the column $(s+d)/2$,
it follows that $(s+d)/2 \notin \la$.
If $s$ is even and $d$ is odd, then $s$ and $s+2d$ are even. In a similar way, $s/2,(s+2d)/2 \notin \la$ by Proposition \ref{prop:dd_f} (c).
\end{proof}

Now we give a path interpretation for the doubled distinct $(s,s+d,s+2d)$-cores.

\begin{thm}\label{thm:dd3}
For coprime positive integers $s$ and $d$, there is a bijection between the sets $\mathcal{DD}_{(s,s+d,s+2d)}$ and 
\begin{enumerate}
\item[(a)] $\mathcal{F}(\frac{s+d+1}{2},-\frac{d}{2} \,;\, \{U\},\{D\})$ if $s$ is odd and $d$ is even;
\item[(b)] $\mathcal{F}(\frac{s+d}{2},-\frac{d+1}{2} \,;\, \{U\},\emptyset)$ if both $s$ and $d$ are odd;
\item[(c)] $\mathcal{F}(\frac{s+d-1}{2},-\frac{d+1}{2} \,;\, \{U\},\emptyset)$
 if $s$ is even and $d$ is odd.
\end{enumerate}
\end{thm}

\begin{proof}
Part (a) comes from Theorem \ref{thm:barcore} (a). Parts (b) and (c) are followed by Propositions \ref{prop:f_initial} and \ref{prop:dd_f}. Note that the length of the corresponding paths in parts (b) and (c) are different than the original setting. Since parts (b) and (c) in Proposition \ref{prop:dd_f} give only one option for the value of $f$ at the second last step, we no longer need to extend the corresponding path to the end point.

\end{proof}

\subsection{$(s,s+d,s+2d)$-CSYDs}

We recall that for even $s$, $\la$ is an $s$-CSYD if and only if $\la$ is an $\overline{s}$-core and $3s/2 \notin \la$.
\begin{prop}\label{prop:csyd_f}
For a strict partition $\la$ such that $S(\la)$ is an $(s,s+d,s+2d)$-CSYD, the $(\overline{s+d},d)$-abacus function $f$ of $\la$ satisfies the following.
\begin{enumerate}
    \item [(a)] If $s$ is odd and $d$ is even, then $f(\frac{s+d-1}{2})\in\{-\frac{d+2}{2},-\frac{d}{2}\}$.
    \item [(b)] If $s$ and $d$ are both odd, then $f(\frac{s+d}{2}) \in  \{-\frac{d+1}{2},-\frac{d-1}{2}\}$. In addition, $f(\frac{s+d-2}{2})=-\frac{d+1}{2}$ when $f(\frac{s+d}{2})=-\frac{d-1}{2}$.
    \item [(c)] If $s$ is even and $d$ is odd, then $f(\frac{s+d-1}{2}), f(\frac{s+d-3}{2}) \in \{ -\frac{d+3}{2}, -\frac{d+1}{2},  -\frac{d-1}{2}\}$. 
    
\end{enumerate}
\end{prop}

\begin{proof}
\begin{enumerate}
    \item [(a)] It also follows from Proposition \ref{prop:barf} (a) since we do not need to consider the additional property of an $S(\la)$.
    \item [(b)] From the proof of Proposition \ref{prop:barf} (b), we have $(3s+3d)/2\notin \la$ for an $(\ols{s\phantom{d}},\overline{s+d},\overline{s+2d})$-core partition $\la$. Therefore, $\la$ is an $(\ols{s\phantom{d}},\overline{s+d},\overline{s+2d})$-core partition if and only if $S(\la)$ is an $(s,s+d,s+2d)$-CSYD for odd numbers $s$ and $d$.
    \item [(c)] Let $(a,b)=(-(d+3)/2,(s+d-3)/2)$. By the proof of Proposition \ref{prop:barf} (c) we have $f(b+1)=a,a+1$, or $a+2$. Note that positions $(a,b)$, $(a+1,b)$,
    $(a+2,b)$, and $(a+3,b)$ are labeled by $-(3s+6d)/2$, $-(s+4d)/2$, $(s-2d)/2$, and $3s/2$ respectively. 
    Since $3s/2,(3s+6d)/2 \notin \la$ by Proposition \ref{prop:CSYD} (b), there is at most one bead labeled by $(s-2d)/2$ or $-(s+4d)/2$ in column $b$. Hence, $f(b)=a,a+1$, or $a+2$.
\end{enumerate}
\end{proof}

Again, we construct a bijection between the set of $(s,s+d,s+2d)$-CSYDs and the set of functions satisfying the conditions in Propositions \ref{prop:f_initial} and \ref{prop:csyd_f}.

\begin{prop}
For coprime positive integers $s$ and $d$, let $f$ be a function that satisfies Propositions \ref{prop:f_initial} and \ref{prop:csyd_f}. If $\la$ is a strict partition such that $f$ is the $(\overline{s+d},d)$-abacus function of $\la$, then $S(\la)$ is an $(s,s+d,s+2d)$-CSYD.
\end{prop}

\begin{proof}
Similar to Proposition \ref{prop:dd_inverse}, it is sufficient to show that $\la$ satisfies Proposition \ref{prop:CSYD} (b). Also, we do not need to check the additional condition when $s$ is odd and $d$ is even. If $s$ and $d$ are both odd, by Proposition \ref{prop:csyd_f} (b), there is at most one bead labeled by $(s+d)/2$ in column $(s+d)/2$. Since no columns but the column $(s+d)/2$ has labels whose absolute values are $(3s+3d)/2$, it follows that $(3s+3d)/2 \notin \la$. If $s$ is even and $d$ is odd, then only the column $(s+d-3)/2$ has positions labeled by $-(3s+6d)/2$ and $3s/2$. Since there is at most one bead being labeled by $(-s+4d)/2$ or $(s-2d)/2$ in column $(s+d-3)/2$ by Proposition \ref{prop:csyd_f} (c), we have $3s/2,(3s+6d)/2 \notin \la$. It completes the proof.
\end{proof}

Similarly, we give a path interpretation for $(s,s+d,s+2d)$-CSYDs.

\begin{thm}\label{thm:csyd3}
For coprime positive integers $s$ and $d$, there is a bijection between the sets $\mathcal{CS}_{(s,s+d,s+2d)}$ and
\begin{enumerate}
\item[(a)]
$\mathcal{F}(\frac{s+d+1}{2},-\frac{d}{2} \,;\, \{U\},\{D\})$ if $s$ is odd and $d$ is even;
\item[(b)] $\mathcal{F}(\frac{s+d+2}{2},-\frac{d+1}{2} \,;\, \{U\},\{FD,DD,U\})$ if both $s$ and $d$ are odd;
\item[(c)] 
$\mathcal{F}(\frac{s+d+1}{2},-\frac{d+1}{2} \,;\, \{U\},\{UU,DD\})$ if $s$ is even and $d$ is odd.
\end{enumerate}
\end{thm}

\begin{proof}
Parts (a) and (b) follow from Theorem \ref{thm:barcore}. Now we need to construct a bijection for the set $\mathcal{CS}_{(s,s+d,s+2d)}$ when $s$ is even and $d$ is odd. Until the second last step of the corresponding free Motzkin paths, the paths should be in one of the following sets:
\begin{align*}
&\mathcal{F}\left((s+d-1)/2,-(d-1)/2 \,;\, \{U\},\{D\}\right),\\ 
&\mathcal{F}\left((s+d-1)/2,-(d+1)/2 \,;\, \{U\},\emptyset\right),\\ 
&\mathcal{F}\left((s+d-1)/2,-(d+3)/2 \,;\, \{U\},\{U\}\right).
\end{align*}
By adding the end point of the free Motzkin path, we get the statements. 

\end{proof}

\subsection{Enumerating $(s,s+d,s+2d)$-core partitions}
In this subsection we give a proof of Theorem~\ref{thm:unifying}. We begin with a useful lemma.

\begin{lem}\label{lem:path1}
Let $a$ and $b$ be positive integers. 
\begin{enumerate}
\item[(a)] The total number of free Motzkin paths of type $(a+b,-b)$ for which starts with either a down or a flat step is given by
\[
|\mathcal{F}(a+b,-b \,;\, \{U\},\emptyset)|=\sum_{i=0}^{a}\binom{a+b-1}{\lfloor i/2 \rfloor, b+\lfloor (i-1)/2\rfloor, a-i}.
\]
\item[(b)] The total number of free Motzkin paths of type $(a+b,-b)$ for which starts with either a down or a flat step and ends with either a up or a flat step is
\[
|\mathcal{F}(a+b,-b \,;\, \{U\},\{D\})|=\sum_{i=0}^{a-1}\binom{a+b-2}{\lfloor i/2 \rfloor}\binom{a+b-1-\lfloor i/2 \rfloor}{a-i-1}.
\]
\item[(c)] The total number of free Motzkin paths of type $(a+b,-b)$ for which starts with either a down or a flat step and ends with either a down or a flat step is
\[
|\mathcal{F}(a+b,-b \,;\, \{U\},\{U\})|=\sum_{i=0}^{a}\binom{a+b-2}{\lfloor i/2 \rfloor}\binom{a+b-1-\lfloor i/2 \rfloor}{a-i}.
\]
\end{enumerate}
\end{lem}

\begin{proof}
\begin{enumerate}
\item[(a)] The number of free Motzkin paths of type $(a+b,-b)$ having $k$ up steps (so that it has $b+k$ down steps and $a-2k$ flat steps) for which starts with a down (resp. flat) step is $\binom{a+b-1}{k,b+k-1,a-2k}$ (resp. $\binom{a+b-1}{k,b+k,a-(2k+1)}$). Hence, the total number of free Motzkin paths of type $(a+b,-b)$ for which starts with either a down or a flat step is 
\[
\sum_{k=0}^{\lfloor a/2 \rfloor}\binom{a+b-1}{k,b+k-1,a-2k}
+\sum_{k=0}^{\lfloor (a-1)/2 \rfloor}\binom{a+b-1}{k,b+k,a-(2k+1)},
\]
which can be written as in the statement.
\item[(b)] Note that
$|\mathcal{F}(a+b,-b\,;\,\{U\},\{D\})|$ is equal to the sum of the two values, which are given by (a), 
\begin{align*}
|\mathcal{F}(a+b-1,-b \,;\, \{U\},\emptyset)|&=\sum_{i=0}^{a-1}\binom{a+b-2}{\lfloor i/2 \rfloor, b+\lfloor (i-1)/2\rfloor, a-i-1},\\
|\mathcal{F}(a+b-1,-b-1 \,;\, \{U\},\emptyset)|&=\sum_{i=0}^{a-2}\binom{a+b-2}{\lfloor i/2 \rfloor, b+\lfloor (i+1)/2\rfloor, a-i-2}.
\end{align*}
Hence, $|\mathcal{F}(a+b,-b\,;\,\{U\},\{D\})|$ is equal to 
\[
\sum_{i=0}^{a-1}\binom{a+b-2}{\lfloor i/2 \rfloor}\left(\binom{a+b-2-\lfloor i/2 \rfloor}{a-i-1}
+\binom{a+b-2-\lfloor i/2 \rfloor}{a-i-2}\right),
\]
which can be written as in the statement.
\item[(c)] Similar to (b), the formula follows.
\end{enumerate}
\end{proof}

For coprime positive integers $s$ and $d$, let $\mathfrak{sc}$, $\mathfrak{bc}$, $\mathfrak{cs}$, and $\mathfrak{dd}$ denote the cardinalities of the sets $\mathcal{SC}_{(s,s+d,s+2d)}$, $\mathcal{BC}_{(s,s+d,s+2d)}$, $\mathcal{CS}_{(s,s+d,s+2d)}$, and $\mathcal{DD}_{(s,s+d,s+2d)}$, respectively.

\begin{proof}[Proof of Theorem~\ref{thm:unifying}.]
\begin{enumerate}
\item[(a)] Recall that for odd $s$ and even $d$, the three sets $\mathcal{BC}_{(s,s+d,s+2d)}$,  $\mathcal{DD}_{(s,s+d,s+2d)},$ and $\mathcal{CS}_{(s,s+d,s+2d)}$ are actually the same by Remark~\ref{rmk:oddoddodd}.  By Theorem \ref{thm:barcore} (a), the set $\mathcal{BC}_{(s,s+d,s+2d)}$ is bijective with $\mathcal{F}((s+d+1)/2,-d/2 \,;\, \{U\},\{D\}).$
By setting $a=(s+1)/2$ and $b=d/2$ in Lemma~\ref{lem:path1}~(b), we obtain a desired formula.

\item[(b)] 
For odd numbers $s$ and $d$, we have $\mathfrak{bc}=\mathfrak{cs}$ by Theorems \ref{thm:barcore} (b) and \ref{thm:csyd3} (b).
By Lemma~\ref{lem:path1}~(a), we get

\begin{align*}
&\left|\mathcal{F}\left(\frac{s+d}{2},-\frac{d+1}{2} \,;\, \{U\},\emptyset\right)\right|=\sum_{i=0}^{(s-1)/2}\binom{(s+d-2)/2}{\lfloor i/2 \rfloor, \lfloor (d+i)/2\rfloor, (s-1)/2-i}, \\
&\left|\mathcal{F}\left(\frac{s+d}{2},-\frac{d-1}{2} \,;\, \{U\},\{F,D\}\right)\right|=\left|\mathcal{F}\left(\frac{s+d-2}{2},-\frac{d+1}{2} \,;\, \{U\},\emptyset\right)\right|\\
& \hspace{54.5mm} =\sum_{i=0}^{(s-3)/2}\binom{(s+d-4)/2}{\lfloor i/2 \rfloor, \lfloor (d+i)/2\rfloor, (s-3)/2-i}.
\end{align*}
As in the proof of Theorem \ref{thm:barcore}, $\mathfrak{bc}$ is equal to the sum of these two terms, which can be written as follows. 
\[
\mathfrak{bc}=\mathfrak{cs}=\sum_{i=0}^{(s-1)/2}\binom{(d-1)/2+i}{\lfloor i/2 \rfloor}\left( \binom{(s+d-2)/2}{(d-1)/2+i} + \binom{(s+d-4)/2}{(d-1)/2+i}\right).
\]

\item[(c)] By Theorem \ref{thm:barcore} (c), the set $\mathcal{BC}_{(s,s+d,s+2d)}$ is bijective with the set $\mathcal{F}((s+d+1)/2,-(d+1)/2 \,;\, \{U\},\emptyset)$ for even $s$ and odd $d$.
By Lemma~\ref{lem:path1}~(a), 
\[
\mathfrak{bc}=\sum_{i=0}^{s/2}\binom{(s+d-1)/2}{\lfloor i/2 \rfloor, (d+1)/2+\lfloor (i-1)/2\rfloor, s/2-i}.
\]
Now we consider the set $\mathcal{CS}_{(s,s+d,s+2d)}$. As in the proof of Theorem \ref{thm:csyd3},  $\mathfrak{cs}=|\mathcal{F}_1|+|\mathcal{F}_2|+|\mathcal{F}_3|$, where
\begin{align*}
\mathcal{F}_1&\coloneqq\mathcal{F}\left(\frac{s+d-1}{2},-\frac{d-1}{2} \,;\, \{U\},\{D\}\right)\!,\\ \mathcal{F}_2&\coloneqq\mathcal{F}\left(\frac{s+d-1}{2},-\frac{d+1}{2} \,;\, \{U\},\emptyset\right)\!,\\ \mathcal{F}_3&\coloneqq\mathcal{F}\left(\frac{s+d-1}{2},-\frac{d+3}{2} \,;\, \{U\},\{U\}\right)\!.
\end{align*}
From Lemma~\ref{lem:path1}, we obtain that
\begin{align*}
|\mathcal{F}_2|&=\sum_{i=0}^{(s-2)/2}\binom{(s+d-3)/2}{\left\lfloor i/2 \right\rfloor} \binom{(s+d-3)/2-\left\lfloor i/2 \right\rfloor}{(s-2)/2-i},\\
|\mathcal{F}_1|+|\mathcal{F}_3|&=\sum_{i=0}^{(s-2)/2}\binom{(s+d-5)/2}{\left\lfloor i/2 \right\rfloor} \binom{(s+d-1)/2-\left\lfloor i/2 \right\rfloor}{(s-2)/2-i},
\end{align*}
which completes the proof.

\item[(d)] Theorem \ref{thm:dd3} (b) and (c), and Lemma \ref{lem:path1} give an expression of $\mathfrak{dd}$ depending on the parity of $s$. By manipulating binomial terms, one can combine two expressions into one.

\end{enumerate}
\end{proof}

\begin{rem} 
From the path constructions, we compare the sizes among them.

\begin{enumerate}
\item[(a)] If $s$ is odd and $d$ is even, then $\mathfrak{sc}<\mathfrak{bc}=\mathfrak{cs}=\mathfrak{dd}$.
\item[(b)] If both $s$ and $d$ are odd, then $\mathfrak{sc}=\mathfrak{dd}<\mathfrak{bc}=\mathfrak{cs}$.
\item[(c)] If $s$ is even and $d$ is odd, then $\mathfrak{dd}<\mathfrak{cs}<\mathfrak{sc}=\mathfrak{bc}$.
\end{enumerate}
\end{rem}

\section*{Acknowledgments}
Hyunsoo Cho was supported by the National Research Foundation of Korea(NRF) grant funded by the Korea government(MSIT) (No. 2021R1C1C2007589) and the Ministry of Education (No. 2019R1A6A1A11051177). JiSun Huh was supported by the National Research Foundation of Korea(NRF) grant funded by the Korea government(MSIT) (No. 2020R1C1C1A01008524). Jaebum Sohn was supported by the National Research Foundation of Korea (NRF) grant funded by the Korea government (MSIT) (NRF-2020R1F1A1A01066216).








\end{document}